\documentclass{article}
\usepackage{graphicx} 
\usepackage[a4paper, total={6in, 9in}]{geometry}

\usepackage{amsfonts}
\usepackage{amsmath}
\usepackage{amssymb}
\usepackage{amsthm}
\usepackage{bbm}
\usepackage{hyperref}
\usepackage{mathtools}

\newtheorem{theorem}{Theorem}
\newtheorem{lemma}[theorem]{Lemma}
\newtheorem{prop}[theorem]{Proposition}

\newtheorem{definition}{Definition}

\newtheorem{example}{Example}

\newcommand{\n}{\mathbb{N}}
\newcommand{\z}{\mathbb{Z}}
\newcommand{\R}{\mathbb{R}}

\newcommand{\rd}{\mathbb{R}^d}

\newcommand{\borel}{\mathcal{B}}

\newcommand{\e}{\mathbb{E}}

\newcommand{\p}{\mathbb{P}}

\newcommand{\abs}[1]{\left|#1\right|}
\newcommand{\norm}[1]{\left\|#1\right\|}
\newcommand{\ind}{\mathbbm{1}}
\newcommand{\dx}{\mathrm{d}}
\newcommand{\dotp}[2]{\left< #1,#2 \right >}

\newcommand{\ppp}{\eta}
\newcommand{\markdist}{\mathbb{Q}}

\newcommand{\as}{a.\,s.\ }
\newcommand{\ie}{i.\,e.\ }
\newcommand{\wrt}{w.\,r.\,t.\ }

\newcommand{\Msettemp}{\mathcal{M}_{\text{temp},\delta}}
\newcommand{\Msetl}{\mathcal{M}_{l,\delta}}

\newcommand{\bfx}{{\bf x}}
\newcommand{\bfy}{{\bf y}}

\newcommand{\hatMsetl}{\Hat{\mathcal{M}}_{l,\delta}}
\newcommand{\Fd}{\mathcal{F}_d}
\newcommand{\FdAp}{\mathcal{F}_d'}
\newcommand{\Cd}{\mathcal{C}(\rd)}

\newcommand{\pd}{\rho}

\newcommand{\tailF}[1]{\mathsf{F}_{#1}}
\newcommand{\tailFnula}[1]{\mathsf{F}^0_{#1}}

\newcommand{\interior}{\operatorname{int}}
\newcommand{\closure}{\operatorname{clo}}

\newcommand{\keywords}[1]{\par\noindent\textbf{Keywords: }#1}

\newcommand{\msc}[1]{\par\noindent\textbf{MSC: }#1}

\title{Mixing Properties of Random Laguerre Tessellations}

\author{Zbyněk Pawlas\thanks{Department of Probability and Mathematical Statistics, Faculty of Mathematics and Physics, Charles University, Czechia,  e-mails: \textit{zbynek.pawlas@matfyz.cuni.cz, petrakova@karlin.mff.cuni.cz}.}\,\, and  Martina Švarc Petráková\footnotemark[1]}

\date{}

\begin{document}

\maketitle

\begin{abstract}
  In this paper, we study the random Laguerre tessellation, a weighted generalization of the Voronoi tessellation, generated by a general stationary marked point process. We first derive a nearly optimal sufficient condition on the generating marked point process which ensures that the resulting random Laguerre tessellation is well-defined, utilising the concept of tempered configurations to handle potentially unbounded weights. We then investigate how the three mixing properties -- ergodicity, mixing and $\alpha$-mixing -- of the generating marked point process are preserved for the corresponding random Laguerre tessellation. Our approach combines standard approximation arguments with the properties of tempered configurations and the measurability of the Laguerre mapping.
\end{abstract}

\msc{60D05, 37A25, 60G55} 

\section{Introduction}
Random tessellations constitute a prominent class of stochastic geometry models with applications in materials science, biology and other fields. Among other models, random Voronoi tessellations, their weighted generalizations and the corresponding dual tessellations have been studied extensively, see \cite{ar:RJ25} for a recent comprehensive overview. 
In this paper, we focus on Laguerre tessellations, which form the most widely used weighted generalization of Voronoi tessellations.

Let $z \in \rd$ and $\bfx = (x,m) \in \rd \times \R$. We call $x$ the nucleus and $m$ the weight or mark. The power distance of a point $z$ to a weighted point $\bfx$ is defined as
\begin{equation*} 
    \pd(z,\bfx) \coloneqq \norm{x-z}^2 + m. 
\end{equation*} 

Let $\varphi \subseteq\rd \times \R$ be an at most countable set of weighted points (called the set of generators). For $\bfx \in \varphi$ define the Laguerre cell associated with $\bfx$ as 
\begin{equation*} 
L(\bfx,\varphi) \coloneqq \{z \in \rd: \pd(z,\bfx) \leq \pd(z,\bfy) \; \forall \bfy \in \varphi\},
\end{equation*}
and the Laguerre diagram generated by $\varphi$ as 
\begin{equation}\label{def:LaguerreDiagram}
L(\varphi) \coloneqq \{L({\bf x},\varphi): {\bf x} \in \varphi,\, \interior(L({\bf x},\varphi)) \neq \emptyset\},
\end{equation}
where $\interior(\cdot)$ denotes the interior of a Borel subset of $\rd$. For a suitable set $\varphi$ of generators, the Laguerre diagram $L(\varphi)$ is a tessellation of $\rd$, \ie a locally finite partition of $\rd$ into compact convex sets with disjoint non-empty interiors. If all weights in $\varphi$ are equal, then the cells are determined based on the Euclidean distance to the nuclei and $L(\varphi)$ becomes the Voronoi tessellation. A random Laguerre tessellation is obtained by taking the set of generators to be a suitable stationary marked point process $\ppp$ on $\rd \times \R$ with finite intensity. 

The first systematic study of random Laguerre tessellations was carried out in \cite{ar:LZ08}, focusing primarily on the Poisson--Laguerre tessellation, \ie the Laguerre tessellation generated by a stationary marked Poisson point process. Several geometric characteristics of typical $k$-faces, as well as various limiting results, were derived therein. More recently, Poisson--Laguerre tessellations generated by a (not necessarily marked) Poisson point process on $\rd \times \R$ were investigated in \cite{ar:GWL25, ar:GWL26}, focusing on the sectional properties and convergence of the tessellation and its typical cells. This work followed the series of papers \cite{ar:GKT22a, ar:GKT22b, ar:GKT22c, ar:GKT23}, where three particular models of Poisson--Laguerre tessellation (called $\beta$-, $\beta'$- and Gaussian Voronoi tessellation) and mainly their dual tessellations were studied. The sectional properties of these three models were subsequently analysed in \cite{ar:GKT24}. Statistical inference for Poisson--Laguerre tessellations was considered in \cite{ar:FPY20, ar:JJV25}. The non-Poissonian case has received considerably less attention. Gibbs--Laguerre tessellations were studied in \cite{ar:CKR24, ar:JS20}, in both cases under the assumption of bounded weights, while their applications as stochastic models were considered in \cite{ar:Schl24, ar:SMB22, ar:Setal21}. 

The aim of this paper is to develop a unified theory of random Laguerre tessellations generated by a general marked point process on $\rd \times \R$. Following closely \cite{ar:LZ08}, we first establish a sufficient assumption on the random set of generators ensuring that the random Laguerre tessellation is well-defined (see Theorem \ref{thm:Main} (a)). By exploiting the concept of tempered configurations from \cite{ar:RZ20}, this condition requires a nearly optimal moment assumption on the stationary mark distribution. Afterwards, we focus on how the different mixing properties of the random set $\ppp$ of generators are preserved for the random Laguerre tessellation (see Theorem \ref{thm:Main} (b) and (c)) generated by $\ppp$. By a mixing property, we mean, roughly speaking, some kind of asymptotic independence of spatially separated parts of the tessellation. In particular, we consider ergodicity, mixing and $\alpha$-mixing of random Laguerre tessellations in this paper. 

The remainder of the paper is organized as follows. In Section \ref{sec:NotationAndResult}, we present the necessary notation followed up by our three main results collected in Theorem \ref{thm:Main}. Section \ref{sec:LaguerreAndTempered} deals with the formal introduction of random Laguerre tessellations -- particularly the measurability of the mapping $\varphi \mapsto L(\varphi)$ in Theorem \ref{thm:MeasurabilityOfL} -- and the employ of the set of tempered configurations within the Laguerre formalism. The proofs of Theorems \ref{thm:Main} and \ref{thm:MeasurabilityOfL} can be found in Section \ref{sec:Proofs}. We finish our paper by Section \ref{sec:Examples}, in which we present several examples of marked point processes for which the random Laguerre tessellation is well-defined and discuss its mixing properties.

\section{Notation and Main Result}\label{sec:NotationAndResult}

Let $d \geq 2$ be our dimension and denote by $\lambda^d$ the $d$-dimensional Lebesgue measure on $\rd$. Let $ \borel(\rd)$ be the Borel $\sigma$-algebra on $\rd$ and denote by $B(x,r)$ the closed ball with centre $x \in \rd$ and radius $r\geq 0$. We frequently use the notation $B^d\coloneqq B(0,1)$ and $rB^d\coloneqq B(0,r)$. Throughout this text, we consider to have an underlying probability space $(\Omega, \mathcal{A}, \p)$. 

First, we present necessary concepts from the point process theory, see \cite{bo:DVJ03a,bo:DVJ08b,bo:LP17,bo:SW08} for further background.
A point process on $\rd$ is a random element in the space $(\mathsf{N}(\rd),\mathcal{N}(\rd))$ of all locally finite counting measures on $(\rd,\borel(\rd))$ equipped with the smallest $\sigma$-algebra for which all evaluation mappings $\varphi \mapsto \varphi(B)$, $B \in \borel(\rd )$, are measurable. A point process $\ppp'$ on $\rd$ is called simple if $\p(\ppp' \in \mathsf{N}_s(\rd))=1$, where $\mathsf{N}_s(\rd)\coloneqq \{\varphi \in \mathsf{N}(\rd): \varphi(\{x\}) \leq 1 \; \forall x \in \rd\}$. As usual, we identify the simple point process $\ppp'$ with its support. 

A point process $\ppp$ on $\rd \times \R$ is called a marked point process if $\p(\ppp \in \mathsf{N}_{m}(\rd\times \R)) = 1$, where $\mathsf{N}_{m}(\rd\times \R)\coloneqq \{\varphi \in \mathsf{N}(\rd\times \R): \varphi'(\cdot)\coloneqq \varphi(\cdot \times \R) \in \mathsf{N}(\rd)\}$, \ie if $\ppp'$ is a point process on $\rd$. We call $\ppp'$ the ground process. A marked point process is simple if its ground process is simple. 
We say that a marked point process $\ppp$ on $\rd \times \R$ is stationary if $T_z(\ppp)\overset{D}{=}\ppp$ for any $z \in \rd$, where $T_z: \mathsf{N}_m(\rd\times \R) \rightarrow \mathsf{N}_m(\rd\times \R)$ is the shift mapping defined as 
\begin{equation*}
    T_z(\varphi)\coloneqq \varphi + z\coloneqq \sum_{(x,m) \in \varphi} \delta_{(x+z,m)}.
\end{equation*} 
If, in addition, the intensity measure of the ground process $\ppp'$ is locally finite (\ie finite on compacts), then there exist $\gamma > 0$ and a probability distribution $\markdist$ on $\R$ such that for any bounded $B \in \borel(\rd)$ and any measurable $A\subseteq\R$,
\begin{equation*}
    \e \ppp(B \times A) = \gamma \lambda^d(B) \markdist(A).
\end{equation*}
 We call $\markdist$ the stationary mark distribution of $\ppp$ and $M \sim \markdist$ the typical mark of $\ppp$. From now on, we assume that $\ppp$ is non-degenerate and simple and that $\ppp'$ has a locally finite intensity measure.

Next, we present three types of mixing properties that are considered in this paper. First, recall the standard definition of the $\alpha$-mixing coefficient between two $\sigma$-algebras $\mathcal{A}_1$ and $\mathcal{A}_2$, 
\begin{equation*} 
    \alpha(\mathcal{A}_1,\mathcal{A}_2)\coloneqq \sup\{\abs{\p(A_1 \cap A_2) - \p(A_1) \p(A_2)}: A_i \in \mathcal{A}_i, i = 1,2 \}.
\end{equation*}
For further details on $\alpha$-mixing we refer to \cite{bo:Brad07}. Let $\ppp$ be a stationary marked point process on $\rd \times \R$ and denote by $\sigma(\ppp \cap (B \times \R))$ the $\sigma$-algebra generated by the restriction of $\ppp$ to $B \times \R$, $B \in \borel(\rd)$. The $\alpha$-mixing coefficient for $\ppp$ is then defined as
\begin{align*} 
\begin{split}
    \alpha_\ppp(c_1,c_2;\Delta)\coloneqq \sup_{B_i \in \borel(\rd),\, \text{dist}(B_1,B_2)\geq \Delta,\, \lambda_d(B_i)\leq c_i,\, i = 1,2} \alpha(\sigma(\ppp \cap (B_1 \times \R)),\sigma(\ppp \cap (B_2 \times \R))),
\end{split}
\end{align*}
where $c_1,c_2 \in (0,\infty]$, $\Delta > 0$ and $\operatorname{dist}(B_1,B_2)\coloneqq \inf\{\norm{x_1-x_2}: x_i \in B_i,\, i = 1,2\}$. 

\begin{definition} \label{def:MixingMPP}
Let $\ppp$ be a stationary marked point process on $\rd \times \R$. We say that it is 
\begin{itemize}
    \item[(i)]ergodic, if for all $A,B \in \mathcal{N}(\rd \times \R)$,
\begin{equation*}
\lim_{a \rightarrow \infty} \frac{1}{\lambda^d((-a,a)^d)} \int_{(-a,a)^d} \p(\ppp\in A, T_z(\ppp) \in B) \,\dx z = \p(\ppp \in A)\p(\ppp \in B),
\end{equation*}
    \item[(ii)]mixing, if for any $A, B \in \mathcal{N}(\rd \times \R)$,
\begin{equation*}
    \lim_{\norm{z}\rightarrow \infty} \p(\ppp \in A, T_z(\ppp) \in B) = \p(\ppp \in A) \p(\ppp \in B),
\end{equation*}
    \item[(iii)] $\alpha$-mixing, if for all $0< c < \infty$,
    \begin{equation*}
    \lim_{\Delta \rightarrow \infty}  \alpha_\ppp(c,\infty;\Delta)=0.
    \end{equation*}
\end{itemize}
\end{definition}

The definitions of ergodicity and mixing follow \cite[Definition 12.3.I]{bo:DVJ08b}, while the notion of $\alpha$-mixing for marked point processes follows \cite{ar:BW19, ar:WG09}. It is well-known that $\alpha$-mixing implies mixing, which in turn implies ergodicity.

To formally define a random tessellation, we work with the following spaces. Let $(E,\mathcal{G}(E))$ be a locally compact topological space with countable base and denote by $\mathcal{B}(E)$ its Borel $\sigma$-algebra. Denote by $\mathcal{F}(E)$ the set of closed subsets of $E$ and let $\mathcal{F}'(E)\coloneqq \mathcal{F}(E)\setminus \{\emptyset\}$. For $A \subseteq E$, define 
\begin{align*}
    \mathcal{F}_A&\coloneqq \{F \in \mathcal{F}(E): F \cap A \neq \emptyset\},\\
    \mathcal{F}^A&\coloneqq \{F \in \mathcal{F}(E): F \cap A = \emptyset\}.
\end{align*}
The Fell topology $\tau_{\mathcal{F}}$ on $\mathcal{F}(E)$ is generated by the subbasis $\{\mathcal{F}^C: C \in \mathcal{C}(E)\} \cup \{\mathcal{F}_G: G \in \mathcal{G}(E)\}$, where  $\mathcal{C}(E)$ is the set of compact subsets of $E$. In particular, the space $(\mathcal{F}(E),\tau_{\mathcal{F}})$ is compact and the space $(\mathcal{F}'(E),\tau_{\mathcal{F}})$ is locally compact and both have a countable base, see \cite[Theorem 12.2.1 and Remark (b)]{bo:SW08}. We denote by $\borel(\mathcal{F}(E))$ the $\sigma$-algebra generated by the Fell topology and analogously $\borel(\mathcal{F}'(E))$.  In this paper, we consider two choices of $E$,  the Euclidean space $\rd$ and the set of its non-empty closed subsets, $\FdAp\coloneqq \mathcal{F}'(\rd)$.

\begin{definition} \label{def:tessellation}
    A tessellation of $\rd$ is a locally finite set $\{C_i\}_{i \in \n}$ of compact convex cells $C_i$, with non-empty pairwise disjoint interiors, satisfying $\cup_{i \in \n} C_i = \rd$. The set of all tessellations of $\rd$ is denoted as $\mathbb{M}_d$.
\end{definition}

Recall that local finiteness means that any compact $C \in \mathcal{C}(\rd)$ intersects only finitely many cells $C_i$. It follows from \cite[Lemma 10.1.2]{bo:SW08} that $\mathbb{M}_d $ is a Borel subset of $\mathcal{F}(\FdAp)$, which ensures the correctness of the following definition. 

\begin{definition} \label{def:randomtessellation}
    A random tessellation is a random element $X$ in the space $\mathcal{F}(\FdAp)$ which satisfies $\p(X \in \mathbb{M}_d)=1$. 
\end{definition}

To define the mixing properties for random tessellations, define the shift operators $\mathcal{T}_z$, $z \in \rd$, on the space $\mathcal{F}(\FdAp)$,
\begin{equation*} 
    \mathcal{T}_z(Z)\coloneqq Z+z\coloneqq\{F+z: F \in Z\}, \, \, Z \in \mathcal{F}(\FdAp).
\end{equation*}
A random tessellation $X$ is stationary if $\mathcal{T}_z(X)\overset{D}{=}X$ for any $z \in \rd$. Let $a > 0$ and define two semi-algebras
\begin{align*}
    \begin{split}
        \tailFnula{a}&\coloneqq\{ \mathcal{F}^{C_0}_{C_1,\dots,C_k}: C_i \in \mathcal{C}(\FdAp) \text{ s.t. } C_i \subseteq \mathcal{F}_{aB^d}, i \in \{0,\dots,k\}, k \in \n_0\}, \\
        \tailFnula{-a}&\coloneqq \{ \mathcal{F}^{C_0}_{C_1,\dots,C_k}: C_i \in \mathcal{C}(\FdAp) \text{ s.t. } C_i \subseteq \mathcal{F}_{\rd \setminus aB^d}, i \in \{0,\dots,k\}, k \in \n_0\}.
    \end{split}
\end{align*}
Then the $\sigma$-algebra $\sigma \{ \tailFnula{a}\}$ contains (roughly speaking) all the information about the sets intersecting the ball $aB^d \subseteq\rd$ and similarly for $\sigma \{ \tailFnula{-a}\}$.

\begin{definition} \label{def:MixTess} Let $X$ be a stationary random tessellation. It is called 
\begin{itemize}
\item[(i)]ergodic, if for all $A,B \in \borel(\mathcal{F}(\FdAp))$,
\begin{equation*}
\lim_{a \rightarrow \infty} \frac{1}{\lambda^d((-a,a)^d)} \int_{(-a,a)^d} \p(X \in A, \mathcal{T}_z(X) \in B) \,\dx z = \p(X \in A) \p(X \in B),
\end{equation*}
    \item[(ii)] mixing, if for any $A, B \in \borel(\mathcal{F}(\FdAp))$,
    \begin{equation*}
    \lim_{\norm{z}\rightarrow \infty} \p(X \in A, \mathcal{T}_z(X) \in B) = \p(X \in A) \p(X \in B),
\end{equation*}
\item[(iii)] $\alpha$-mixing, if for any $a > 0$,
\begin{equation*} 
        \lim_{b \rightarrow \infty} \alpha(\sigma_X(\tailFnula{a}),\sigma_X(\tailFnula{-b})) = 0,
    \end{equation*}
where $\sigma_X(\tailFnula{a})\coloneqq \sigma(\{\{X \in A\}: A \in \tailFnula{a}\})$ and $\sigma_X(\tailFnula{-b})\coloneqq\sigma(\{\{X \in B\}: B \in \tailFnula{-b}\})$. 
\end{itemize}
\end{definition}

This definition follows in spirit the approach of \cite{ar:GKT23, ar:H94}. The difference is that we consider pre-images of events from $\borel(\mathcal{F}(\FdAp))$ defined through the cells of the random tessellation, rather than pre-images of events from $\borel(\FdAp)$ defined through its skeleton (\ie the union of the boundaries of the cells). Again, the mixing properties are presented from the weakest to the strongest.

Mixing properties of several random tessellation models have been studied in the literature. Mixing of the Poisson hyperplane tessellation is discussed in \cite[Chapter 7]{bo:HS24} and also in \cite[Section 10.5]{bo:SW08}, where mixing of the Poisson--Voronoi tessellation and its dual is also considered. Mixing of the STIT tessellation has been proved in \cite{ar:LR11}. The absolute regularity (or $\beta$-mixing) of random tessellations, which is not studied in this paper, was considered in \cite{ar:H94} for the random Voronoi tessellation, in \cite{ar:GKT23} for the dual of special models of Poisson--Laguerre tessellations and in \cite{ar:MN16} for the STIT tessellation.   

We are now ready to present our main theorem, which provides sufficient conditions for the three mixing properties of random Laguerre tessellations. 

\begin{theorem}\label{thm:Main}
    Let $\ppp$ be a stationary marked point process on $\rd \times \R$ with a stationary mark distribution $\markdist$. Assume that there exists $\delta > 0$ such that the typical mark $M \sim \markdist$ satisfies 
\begin{equation} \label{as:momentM_}
\e M_-^\frac{d+\delta}{2}< \infty,
\end{equation} 
where $M_-$ is the negative part of $M$. Then the following implications hold.
\begin{itemize}
    \item[(a)] If $\ppp$ is ergodic, then $L(\ppp)$ is \as a tessellation of $\rd$ and it is also ergodic.
    \item[(b)] If $\ppp$ is mixing, then $L(\ppp)$ is mixing.
    \item[(c)] If $\ppp$ is $\alpha$-mixing, then $L(\ppp)$ is $\alpha$-mixing.
\end{itemize}
\end{theorem}

The proofs of parts (a) and (b) are given in Section \ref{subsec:Proofab}, while the proof of part (c) can be found in Section \ref{subsec:Proofc}. As a possible future direction, we expect that a similar approach could be used to deal with the $\beta$-mixing for random Laguerre tessellations. 

Let us conclude by stating that our proofs do not, in principle, provide any information about the rate of decay of the mixing coefficients, even if such information were available for the random set $\ppp$ of generators. This is due to the employ of the set of tempered configurations. Further study of the properties of this set would therefore be required to be able to say anything about the rate of mixing.

\section{Laguerre Formalism and Tempered Configurations} \label{sec:LaguerreAndTempered}

The goal of this section is to formalize the definition of a random Laguerre tessellation and find sufficient conditions for a marked point process $\ppp$ under which $L(\ppp)$ is well-defined. To this end, we first present the concept of tempered configurations in Section \ref{subsec:temperedness}, while Section \ref{subsec:RandomLaguerreTess} deals with the Laguerre formalism, following closely \cite{diz:Lautensack, ar:LZ08}.

Let us leave a small note on the difference between our definition \eqref{def:LaguerreDiagram} of the Laguerre diagram and the (standard) setting of \cite{ar:LZ08}, in which the Laguerre diagram is defined as the collection of all non-empty Laguerre cells, rather than only those with non-empty interiors. If $\varphi$ is in general position (see \cite[Section 3]{ar:LZ08} for the definition), then every Laguerre cell is either empty or has dimension $d$, and both definitions coincide. However, the assumption of general position is not needed in the remainder of our paper. We therefore decided to omit this assumption and rather solve the problem of lower-dimensional Laguerre cells within the definition of the Laguerre diagram.

\subsection{Tempered Configurations} \label{subsec:temperedness}
To be able to handle unbounded weights, we consider the definition of tempered configurations from \cite[Section 2.2]{ar:RZ20}. Let $ \delta > 0$ and denote
\begin{align*}
    \Msetl \coloneqq \{ \varphi \in \mathsf{N}_{s,m}(\rd \times \R) :  \sum_{{(x,m)} \in \varphi\cap (kB^d \times \R)} (1 + \abs{m}^{d+\delta}) \leq l \cdot k^d \text{ for all }k \in \n \},
\end{align*}
where $\mathsf{N}_{s,m}(\rd \times \R)$ is the set of all simple marked measures from $\mathsf{N}_{m}(\rd \times \R)$. The set of tempered configurations with parameter $\delta$ is the set $\Msettemp \coloneqq \bigcup_{l \in \n} \Msetl$. As noted before, we can treat elements of $\Msettemp$ both as locally finite counting measures and as locally finite at most countable subsets of $\rd \times \R$ through the notion of the support of the measure.

Define for any $\varphi \in \mathsf{N}_{s,m}(\rd \times \R)$ the following transformation 
\begin{equation} \label{def:hattrans}
    \hat{\varphi}\coloneqq \sum_{(x,m) \in \varphi} \delta_{(x,\sqrt{m_-})},
\end{equation}
where $m_-$ is the negative part of $m$. For $l \in \n$ and $\delta > 0$ denote
\begin{equation}\label{def:hatMsetl}
    \hatMsetl\coloneqq \{\varphi \in \mathsf{N}_{s,m}(\rd \times \R): \hat{\varphi} \in \Msetl\}.
\end{equation}
The set of tempered configurations allows us to control the norm of the marks. However, for the construction of the Laguerre tessellation, only negative marks with large norm need to be controlled. This motivates the transformation \eqref{def:hattrans}, which allows us to work with tempered configurations without imposing any moment assumptions on the positive part of the typical mark, which is in principle not necessary. 

Before we move to the Laguerre formalism, we present the following basic observations. 
\begin{lemma} \label{lemma:PropertiesTempered}
   Let $l \in \n$ and $\delta > 0$. Then for any $\varphi\in \hatMsetl$ the following implication holds
  \begin{equation}\label{Bound-anypoints}
   (x,m) \in \varphi \implies m > - l ^{\frac{2}{d+\delta}}(\norm{x}+1)^\frac{2d}{d+\delta}.
\end{equation}
Furthermore, denote $k(l) \coloneqq \frac{1}{2} \cdot 2^{\frac{d+\delta}{\delta}} \cdot l^{\frac{1}{\delta}} $, then for all $k \geq k(l)$ we have that
 \begin{equation}\label{Bound-farpoints}
    (x,m) \in \varphi \cap (((2k+1)B^d)^c\times\R) \implies m \geq -(\norm{x}-k)^2.
\end{equation}
\end{lemma}

\begin{proof}
 It follows from \cite[Lemmas 1 and 2]{ar:RZ20} that for any $ \Tilde{\varphi} \in \Msetl$ the following implication holds for all $k \geq k(l)$,
\begin{equation*} 
    {(x,m)} \in \Tilde{\varphi}\cap (((2k+1)B^d)^c \times \R) \implies B(x,m) \cap \interior(B(0, k)) = \emptyset.
\end{equation*}
This implies property \eqref{Bound-farpoints}. The property \eqref{Bound-anypoints} follows from the definition of the set $\Msetl$. 
\end{proof}

The property \eqref{Bound-farpoints} is useful for points far from the origin, while \eqref{Bound-anypoints} can be used for the rest. 

\subsection{Random Laguerre Tessellation} \label{subsec:RandomLaguerreTess}

Recall the definition \eqref{def:LaguerreDiagram} of the Laguerre diagram $L(\varphi)$. Let $\varphi \subseteq\rd \times \R$ be an at most countable set of weighted points and denote by $\varphi' \coloneqq \{x \in \rd: \exists m \in \R \text{ such that } (x,m) \in \varphi \}$ the set of all nuclei. We say that $\varphi$ satisfies the regularity conditions if 
\begin{itemize}
    \item[(R1)] for all $z \in \rd$ and all $t \in \R$ we have $\abs{\{(x,m) \in \varphi: \pd(z,(x,m)) \leq t\}} < \infty$,
    \item[(R2)] $\text{conv}(\varphi') = \rd$.
\end{itemize}

As was stated in \cite[Theorem 3.1]{ar:LZ08},  if $\varphi \subseteq\rd \times (-\infty,0]$ satisfies the regularity conditions, then $L(\varphi)$ is a well-defined tessellation of $\rd$. The same arguments hold also for the case with general real-valued weights. 

\begin{prop} \label{prop:Tessellation}
    Assume that $\varphi \subseteq\rd \times \R$ is a countable set that satisfies the regularity conditions. Then the Laguerre diagram $L(\varphi)$ is a tessellation of $\rd$ (in terms of Definition \ref{def:tessellation}). 
\end{prop}

The proof of this proposition is a retelling of the proofs of \cite[Propositions 2.2.2, 2.2.4 and 2.2.5 and Lemma 2.2.3]{diz:Lautensack}, the only change being the different parametrization for the power distance.  

To define a random Laguerre tessellation, we first comment on the measurability of the Laguerre mapping $L: \varphi \mapsto L(\varphi)$. To the best of our knowledge, such claim, albeit indisputable in its nature, has not been formally proved previously for Laguerre or even Voronoi case. The corresponding claim for hyperplane tessellation was proved in \cite[Section 2.5]{bo:HS24}. 

Recall the notation $\mathsf{N}_{s,m}(\rd \times \R)$ for the space of simple marked measures from $\mathsf{N}(\rd \times \R)$ and let $\mathcal{N}_{s,m}(\rd \times \R)$ denote the standard trace $\sigma$-algebra on this space. 
\begin{theorem}\label{thm:MeasurabilityOfL}
The mapping $L: \mathsf{N}_{s,m}(\rd \times \R) \rightarrow \mathcal{F}(\FdAp)$ defined in \eqref{def:LaguerreDiagram}, which assigns a weighted point configuration to its Laguerre diagram, is measurable w.r.t.\,\,$\mathcal{N}_{s,m}(\rd \times \R)$ and $\borel(\mathcal{F}(\FdAp))$. 
\end{theorem}

The proof of Theorem \ref{thm:MeasurabilityOfL} can be found in Section \ref{subsec:MeasurabilityProof}. We can now formally define the random Laguerre tessellation as a special model of random tessellation. 

\begin{definition}
    Let $\ppp$ be a simple marked point process on $\rd \times \R$ such that $\p(L(\ppp) \in \mathbb{M}_d)=1$. Then it is called an admissible random Laguerre generator and $L(\ppp)$ is called the random Laguerre tessellation generated by $\ppp$.
\end{definition}

In order to check the admissibility of a random set of generators, we verify that it \as satisfies the regularity conditions. As it turns out, it is enough to assume stationarity and check that we can restrict ourselves to the set of tempered configurations. Recall the notation \eqref{def:hattrans}.

\begin{lemma} \label{lemma:ExistenceLagTess}
    Let $\ppp$ be a stationary marked point process such that $\p(\hat{\ppp} \in \Msettemp)=1$, then $\ppp$ is an admissible random Laguerre generator. 
\end{lemma}

\begin{proof}
    According to Proposition \ref{prop:Tessellation}, it is enough to show that $\p(\ppp \text{ satisfies (R1), (R2)})=1$.  Since $\text{supp}(\ppp')$ is a stationary random closed set that is \as non-empty, it follows from \cite[Theorems 2.4.3 and 2.4.4]{bo:SW08} that $\text{conv}(\ppp')\coloneqq \text{conv}(\text{supp}(\ppp')) =  \rd$ \as

To show the first regularity condition, let $\varphi \in \mathsf{N}_{s,m}(\rd \times \R)$ be such that $\hat{\varphi} \in \Msettemp$ and denote by $l(\hat{\varphi})$ the smallest $l$ such that $\hat{\varphi} \in \Msetl$. Take $z \in \rd$ and $t \in \R$. We want to prove that $\abs{\{(x,m) \in \varphi: \pd(z,(x,m)) \leq t\}} < \infty.$ Clearly, it is enough to consider $t \geq 0$, since for any $t < 0$,
    \begin{equation*}
       \{(x,m) \in \varphi: \pd(z,(x,m)) \leq t\} \subseteq \{(x,m) \in \varphi: \pd(z,(x,m)) \leq 0\}.
    \end{equation*}
     
First take $(x,m) \in \varphi \cap (B(z,\sqrt{t})^c \times [0,\infty))$, then clearly $\pd(z,(x,m)) = \norm{x-z}^2+m > t$ and $\abs{\varphi \cap (B(z,\sqrt{t}) \times [0,\infty))} < \infty$ since $\varphi$ is locally finite. To deal with points with negative marks, we employ the tempered property \eqref{Bound-farpoints}. 

Fix $k \geq \max \{k(l(\hat{\varphi})),2\norm{z},\sqrt{t+1}\}$ and take $(x,m) \in \varphi \cap (B(0,2k+1)^c \times (-\infty,0))$, then 
\begin{align*}
    \pd(z,(x,m)) = \norm{x-z}^2+m &\overset{\eqref{Bound-farpoints}}{\geq} \norm{x-z}^2- (\norm{x}-k)^2 \\
    &=\norm{x-z}^2 - \norm{x}^2 + 2\norm{x}k - k^2 \\
    &\geq \norm{x}^2-2\norm{x}\norm{z} + \norm{z}^2 - \norm{x}^2 + 2\norm{x}k - k^2\\
    &\geq 2\norm{x}(k-\norm{z} ) - k^2 \geq 2(2k+1)(k-\norm{z} ) - k^2 \\
    &\geq 3k^2 - 4k\norm{z} \geq 3k^2-2k^2= k^2 > t. 
\end{align*}

Altogether, we get
\begin{equation*}
    \{(x,m) \in \varphi: \pd(z,(x,m)) \leq t\} \subseteq \varphi \cap ((B(z,\sqrt{t}) \times [0,\infty)) \cup (B(0,2k+1) \times (-\infty,0))),
\end{equation*}
or in other words that $\varphi$ satisfies (R1). Therefore, $\p(\hat{\ppp} \in \Msettemp)=1$ implies  $\p(\ppp\text{ satisfies (R1)}) = 1$ and the proof is complete. 
\end{proof}

\section{Proofs} \label{sec:Proofs}

This section consists of the proofs of our main results. In Section \ref{subsec:MeasurabilityProof}, we present the proof of the measurability claim from Theorem \ref{thm:MeasurabilityOfL}, Section \ref{subsec:Proofab} contains the proof of the ergodicity and mixing implications (a) and (b) from Theorem \ref{thm:Main}, while the $\alpha$-mixing implication (c) is proved in Section \ref{subsec:Proofc}. 

\subsection{Proof of Theorem \ref{thm:MeasurabilityOfL}} \label{subsec:MeasurabilityProof}

First, we recall a well-known result about the measurable enumeration of atoms of locally finite counting measures.

\begin{lemma} \label{lemma:MeasEnumeration}
    Let $k \in \n$. There exist measurable mappings $\pi_n : \mathsf{N}(\R^k)\rightarrow \R^k$, $n \in \n$, such that for any $\varphi \in \mathsf{N}(\R^k)$ we have $\varphi = \sum_{n = 1}^{\varphi(\R^k)}\delta_{\pi_n(\varphi)}$. In particular, the mappings $\Pi_n: \R^k \times \mathsf{N}(\R^k) \rightarrow (\R^k )^{n+1}$,
    \begin{equation}\label{def:TildePinmapping}
        \Pi_n(x,\varphi)\coloneqq(x,\pi_1(\varphi),\dots,\pi_n(\varphi)),
    \end{equation}
    are measurable for all $n \in \n$ w.r.t. $\borel(\R^k) \otimes \mathcal{N}(\R^k)$ and $\borel((\R^k)^{n+1})$.
\end{lemma}

\begin{proof}
    For the first assertion, see \cite[Proposition 6.3]{bo:LP17}. The measurability of $\Pi_n$ follows easily. 
\end{proof}

The key to proving Theorem \ref{thm:MeasurabilityOfL} is to realize that 
\begin{equation*}
    L(\varphi) = \{L(\pi_i(\varphi),\varphi): \interior(L(\pi_i(\varphi),\varphi)) \neq \emptyset\}_{i = 1}^{\varphi(\rd \times \R)}.
\end{equation*}

We must therefore first prove the measurability of $(\bfx,\varphi)\mapsto L(\bfx,\varphi)$. Since we later want to disregard cells $L(\bfx,\varphi)$ with empty interiors, we must also distinguish such pairs $(\bfx,\varphi)$ by assigning them an arbitrary set. 

First, we fix some notation. Let $\bfx=(x,m_x)\in\rd \times \R$, $\bfy=(y,m_y)\in\rd \times \R$ and define
\begin{equation*}
        H_{\bfx \leq \bfy}\coloneqq \{z \in \rd: \pd(z,\bfx) \le \pd(z,\bfy) \}=\{z \in \rd: \|z-x\|^2+m_x \le \|z-y\|^2+m_y \}.
\end{equation*}
It can be easily seen that for $\bfx \neq \bfy$ the set $H_{\bfx \leq \bfy}$ is a closed half-space in $\rd$, since 
\begin{equation*}
    H_{\bfx \leq \bfy} 
    = \{z \in \rd:  2\dotp{z}{y-x} \le \|y\|^2 - \|x\|^2 +m_y - m_x\},
\end{equation*}
and that the Laguerre cell is an intersection of such half-spaces. This inspires the following lemma. Denote $\Fd\coloneqq \mathcal{F}(\rd)$.

\begin{lemma}\label{lemma:MeasOfHyperplaneIntersection}
    Let $n \in \n$. Then the mapping $h_n: (\rd \times \R)^{n+1} \rightarrow \Fd$, defined by 
    \begin{equation} \label{def:hnmappinh}
        h_n(\bfx,\bfy_1,\dots,\bfy_n)\coloneqq \bigcap_{i = 1}^n H_{\bfx \leq \bfy_i},
    \end{equation}
    is measurable \wrt\,$\borel((\rd \times \R)^{n+1})$ and $\borel(\Fd)$. 
\end{lemma}

\begin{proof}
    It is enough to check the measurability on the sets $\mathcal{F}^C$, where $C \in \Cd$, since they generate the $\sigma$-algebra $\borel(\Fd)$. 
    Fix $C \in \Cd$, then 
    \begin{align*}
        \{(\bfx,\bfy_1,&\dots,\bfy_n) \in (\rd \times \R)^{n+1}: \bigcap_{i = 1}^n H_{\bfx \leq \bfy_i} \cap C = \emptyset\} \\
        &  = \{(\bfx,\bfy_1,\dots,\bfy_n) \in (\rd \times \R)^{n+1}: \rho(z,\bfx) > \min_{i = 1,\dots,n} \rho(z,\bfy_i)\; \forall z \in C\}\\
        &  = \bigcup_{k = 1}^\infty\{(\bfx,\bfy_1,\dots,\bfy_n) \in (\rd \times \R)^{n+1}: f(\bfx,\bfy_1,\dots,\bfy_n) \geq \frac{1}{k}\},
    \end{align*}
    where $f(\bfx,\bfy_1,\dots,\bfy_n) \coloneqq \sup_{z \in C} (\rho(z,\bfx) - \min_{i = 1,\dots,n} \rho(z,\bfy_i))$. The measurability follows. 
\end{proof}

Now we prove the measurability of the cell mapping. 

\begin{lemma}\label{lemma:MeasurabilityOfCell}
    Define a mapping $ \Tilde{\Lambda} : (\rd \times \R) \times \mathsf{N}_{s,m}(\rd\times \R) \rightarrow \Fd$ by the formula
    \begin{equation*} 
        \Tilde{\Lambda}(\bfx,\varphi)\coloneqq \begin{cases} L(\bfx,\varphi), &\bfx \in \varphi, \\
        B(0,1), &\bfx \notin \varphi.
        \end{cases}
    \end{equation*}
 Then such a mapping is measurable \wrt\,$\borel(\rd \times \R) \otimes \mathcal{N}_{s,m}(\rd\times \R)$ and $\borel(\Fd)$. Furthermore, the mapping $\Lambda:(\rd \times \R) \times \mathsf{N}_{s,m}(\rd\times \R) \rightarrow \FdAp$ defined by the formula
    \begin{equation} \label{def:MappingLambda}
        \Lambda(\bfx,\varphi)\coloneqq \begin{cases} \Tilde{\Lambda}(\bfx,\varphi), &\mathrm{int}(\Tilde{\Lambda}(\bfx,\varphi)) \neq \emptyset, \\
        B(0,2), & \mathrm{int}(\Tilde{\Lambda}(\bfx,\varphi)) = \emptyset,
        \end{cases}
    \end{equation}
     is measurable \wrt\,$\borel(\rd \times \R) \otimes \mathcal{N}_{s,m}(\rd\times \R)$ and $\borel(\FdAp)$.
\end{lemma}

\begin{proof}
    We start with the mapping $\Tilde{\Lambda}$. As in Lemma \ref{lemma:MeasOfHyperplaneIntersection}, it is enough to check the measurability on the sets $\mathcal{F}^C$, where $C \in \Cd$. We can write
    \begin{align*}
        \begin{split}
            \{(\bfx,\varphi) &\in (\rd \times \R) \times \mathsf{N}_{s,m}(\rd\times \R) : \Tilde{\Lambda}(\bfx,\varphi) \in \mathcal{F}^C\} \\
            & = \{(\bfx,\varphi) \in (\rd \times \R) \times \mathsf{N}_{s,m}(\rd\times \R) : \Tilde{\Lambda}(\bfx,\varphi) \cap C = \emptyset\} \\
            & = (\{(\bfx,\varphi)  : B(0,1) \cap C = \emptyset\} \cap \{(\bfx,\varphi)  : \varphi(\{\bfx\})=0\}) \\
            & \hspace{3cm }\cup (\{(\bfx,\varphi)  : L(\bfx,\varphi) \cap C = \emptyset\} \cap \{(\bfx,\varphi)  : \varphi(\{\bfx\}) \geq 1\}).
        \end{split}
    \end{align*}
Clearly $\{(\bfx,\varphi)  : B(0,1) \cap C = \emptyset\} \in \{\emptyset, (\rd \times \R) \times \mathsf{N}_{s,m}(\rd\times \R)\}$ and we can write
\begin{align*}
    \begin{split}
        \{(\bfx,\varphi)  : \varphi(\{\bfx\})\geq 1\} = \bigcap_{n = 1}^\infty \bigcup_{k = 1}^\infty \left(\{(\bfx,\varphi)  : \varphi(B_{n,k})\geq 1\} \cap \{(\bfx,\varphi)  : \bfx \in B_{n,k}\}\right)
    \end{split}
\end{align*}
for a suitable sequence  $\{\{B_{n,k}\}_{k \in \n}\}_{n \in \n}$ of partitions of $\rd$.

It remains to show that the set $\{(\bfx,\varphi)  : L(\bfx,\varphi) \cap C = \emptyset\}$ is measurable. To do that, we employ Lemmas \ref{lemma:MeasEnumeration} and \ref{lemma:MeasOfHyperplaneIntersection} to show the measurability of its complement. Recall the notation $\Pi_n$ from \eqref{def:TildePinmapping} and $h_n$ from \eqref{def:hnmappinh}. Then
\begin{align*}
    \begin{split}
        \{(\bfx,\varphi)  &: L(\bfx,\varphi) \cap C \neq \emptyset\} \\
        &= \{(\bfx,\varphi)  : \bigcap_{\bfy \in \varphi} H_{\bfx \leq \bfy} \cap C \neq \emptyset\}\\
        &= \bigcup_{k = 1}^{\infty} (\{(\bfx,\varphi)  : h_k(\Pi_k(\bfx,\varphi)) \cap C \neq \emptyset\} \cap \{(\bfx,\varphi): \varphi(\rd \times \R) = k\})\\
        & \hspace{2cm} \cup (\{(\bfx,\varphi): \bigcap_{\bfy \in \varphi} H_{\bfx \leq \bfy} \cap C \neq \emptyset\} \cap \{(\bfx,\varphi): \varphi(\rd \times \R) = \infty\}) \\
        &= \bigcup_{k = 1}^{\infty} (\{(\bfx,\varphi)  : h_k(\Pi_k(\bfx,\varphi)) \cap C \neq \emptyset\} \cap \{(\bfx,\varphi): \varphi(\rd \times \R) = k\})\\
        & \hspace{2cm} \cup (\bigcap_{n=1}^\infty\{(\bfx,\varphi): h_n(\Pi_n(\bfx,\varphi)) \cap C \neq \emptyset\} \cap \{(\bfx,\varphi): \varphi(\rd \times \R) = \infty\}).
    \end{split}
\end{align*}
In the last equality, we use the fact that the intersection of a non-empty decreasing sequence of compact sets contains at least one point. Using the measurability of $\Pi_n$, $h_n$ and the evaluation mappings $\varphi \mapsto \varphi(A)$, where $A$ is any fixed measurable set, we get $\Tilde{\Lambda}$ to be measurable in the sense specified above. 

 Let $\closure(B)$ denote the closure of $B \in \borel(\rd)$. To prove that $\Lambda$ defined in \eqref{def:MappingLambda} is also measurable, recall that the mapping $F \mapsto \closure(F^c)$ from $\Fd$ to $\Fd$ is measurable (see \cite[Theorem 12.2.6 (b)]{bo:SW08}) and that $\interior(F) = (\closure(F^c))^c$. Therefore,
\begin{align*}
    \begin{split}
        \{(\bfx,\varphi) &\in (\rd \times \R) \times \mathsf{N}_{s,m}(\rd\times \R) : \Lambda(\bfx,\varphi) \in \mathcal{F}^C\} \\
        & = (\{(\bfx,\varphi)  : B(0,2) \cap C = \emptyset\} \cap \{(\bfx,\varphi)  : \interior(\Tilde{\Lambda}(\bfx,\varphi))=\emptyset\}) \\
            & \hspace{3cm }\cup (\{(\bfx,\varphi)  : \Tilde{\Lambda}(\bfx,\varphi) \cap C = \emptyset\} \cap \{(\bfx,\varphi)  : \interior(\Tilde{\Lambda}(\bfx,\varphi))\neq\emptyset\})\\
        & = (\{(\bfx,\varphi)  : B(0,2) \cap C = \emptyset\} \cap \{(\bfx,\varphi)  : \closure(\Tilde{\Lambda}(\bfx,\varphi)^c)=\rd\}) \\
            & \hspace{3cm }\cup (\{(\bfx,\varphi)  : \Tilde{\Lambda}(\bfx,\varphi) \cap C = \emptyset\} \cap \{(\bfx,\varphi)  : \closure(\Tilde{\Lambda}(\bfx,\varphi)^c)=\rd\}^c)
    \end{split}
\end{align*}
is a measurable set and the proof is finished.
\end{proof}

We are now ready to prove Theorem \ref{thm:MeasurabilityOfL}.

\begin{proof}[Proof of Theorem \ref{thm:MeasurabilityOfL}]
 It follows from Lemma \ref{lemma:MeasurabilityOfCell} that a mapping $\kappa: \rd \times \R \times \mathsf{N}_{s,m}(\rd \times \R) \rightarrow \mathcal{F}(\FdAp)$ defined as 
 \begin{equation*}
 \kappa(\bfx,\varphi)\coloneqq
     \begin{cases}
         \{\Lambda(\bfx,\varphi)\}, & \text{ if } \Lambda(\bfx,\varphi) \notin \{B(0,1),B(0,2)\}, \\
         \emptyset, & \text{ otherwise}, 
     \end{cases}
 \end{equation*}
 is measurable, since for any $C \in \mathcal{C}(\FdAp)$,
 \begin{align*}
     \begin{split}
         \{(\bfx,\varphi): \kappa(\bfx,\varphi) \in \mathcal{F}^C\}^c&= \{(\bfx,\varphi): \kappa(\bfx,\varphi) \cap C \neq \emptyset\} \\
         &= (\{(\bfx,\varphi): \Lambda(\bfx,\varphi) \in C\} \cap \{ (\bfx,\varphi): \Lambda(\bfx,\varphi) \not\in \{B(0,1),B(0,2)\}\}) \\
         & \hspace{1cm}\cup (\{(\bfx,\varphi): \emptyset \in C\}  \cap \{ (\bfx,\varphi): \Lambda(\bfx,\varphi) \in \{B(0,1),B(0,2)\}\}).
     \end{split}
 \end{align*}
  Since the union map is measurable, see \cite[Theorem 12.2.3]{bo:SW08}, we get that for all $n\in \n$ the mappings $L_1^n(\varphi): \mathsf{N}_{s,m}(\rd \times \R) \rightarrow \mathcal{F}(\FdAp) $, defined by
\begin{equation*}
    L_1^{n}(\varphi) \coloneqq \bigcup_{i = 1}^{\max\{n,\varphi(\rd \times \R)\}} \kappa (\pi_i(\varphi),\varphi),
\end{equation*}
are also measurable (note that $L_1^n(\emptyset) \coloneqq \emptyset$). Consider the mapping $L(\varphi)$ defined in \eqref{def:LaguerreDiagram}. Thanks to our definitions of $\Lambda$ and $\kappa$, we get that $L(\varphi) = \cup_{n = 1}^ \infty L_1^n(\varphi)$ and therefore
\begin{align*}
    \begin{split}
        \{\varphi \in \mathsf{N}_{s,m}(\rd \times \R): L(\varphi) \in \mathcal{F}^C\} &= \{\varphi \in \mathsf{N}_{s,m}(\rd \times \R): L(\varphi) \cap C = \emptyset\} \\
        &= \bigcap_{n = 1}^\infty \{\varphi \in \mathsf{N}_{s,m}(\rd \times \R): L_1^n(\varphi) \cap C = \emptyset\},
    \end{split}
\end{align*}
which finishes the proof of the measurability of $L$. 
\end{proof}

\subsection{Proof of Theorem \ref{thm:Main} (a) and (b)} \label{subsec:Proofab}

The proof of the first part of Theorem \ref{thm:Main} is a straightforward consequence of Lemma \ref{lemma:ExistenceLagTess}, which is in principle a slightly more general claim and could be used instead of the implication in (a), see Section \ref{subsec:OtherEx}. The ergodicity and mixing then follow thanks to Theorem \ref{thm:MeasurabilityOfL}. 

\begin{proof}[Proof of Theorem \ref{thm:Main} (a) and (b)] We start with part (a). Recall the notation \eqref{def:hattrans}. According to Lemma \ref{lemma:ExistenceLagTess}, it is enough to show that $\p(\hat{\ppp} \in \Msettemp)=1$. To do that, we apply the same idea as in the proof of \cite[Lemma 1]{ar:NaseEYSM}. Denote $h(m)\coloneqq 1+(\sqrt{m_-})^{d+\delta}$, then \eqref{as:momentM_}, stationarity and ergodicity together with \cite[Corollary 12.2.V (b)]{bo:DVJ08b} give us that 
    \begin{equation*} 
        \frac{1}{\lambda^d(kB^d)} \sum_{(x,m) \in \ppp \cap (kB^d \times \R)} h(m) \overset{\as}{\underset{k \rightarrow \infty}{ \longrightarrow}} \gamma \int_{\R} h(m)\, \dx \markdist(m) \in (0,\infty).
    \end{equation*}
    This implies that $\sup_{k \in \n} \frac{1}{\lambda^d(B^d) k^d} \sum_{(x,m) \in \ppp \cap (kB^d \times \R)} h(m) < \infty$ a.s., which furthermore implies that $\p(\hat{\ppp} \in \Msettemp)=1$ and therefore $L(\ppp)$ is \as a tessellation of $\rd$. 
    
    To prove the ergodicity, recall the measurability of $L(\cdot)$ proved in Theorem \ref{thm:MeasurabilityOfL}. Thanks to the natural duality of the mappings $T_x$ and $\mathcal{T}_x$, we can write for any $A,B \in \borel(\mathcal{F}(\FdAp))$,
    \begin{align*} 
        \lim_{a \rightarrow \infty} \frac{1}{\lambda^d((-a,a)^d)} &\int_{(-a,a)^d} \p(L(\ppp) \in A, \mathcal{T}_z(L(\ppp)) \in B) \,\dx z  \\
        &=\lim_{a \rightarrow \infty} \frac{1}{\lambda^d((-a,a)^d)} \int_{(-a,a)^d} \p(L(\ppp) \in A, L(T_z(\ppp)) \in B) \,\dx z\\
        &= \lim_{a \rightarrow \infty} \frac{1}{\lambda^d((-a,a)^d)} \int_{(-a,a)^d}  \p(\ppp \in L^{-1}(A), T_z(\ppp) \in L^{-1}(B)) \,\dx z\\\
        &= \p(\ppp \in L^{-1}(A)) \p(\ppp \in L^{-1}(B)) = \p(L(\ppp) \in A) \p(L(\ppp) \in B),
    \end{align*}
 which finishes the proof of part (a). The mixing property claimed in part (b) would be proved analogously. 
\end{proof}  

In particular, a mixing random tessellation $X$ satisfies
\begin{equation} \label{cond:mixingproperty}
 \lim_{\norm{x} \rightarrow \infty} \p(X\cap C_1 = \emptyset,X\cap (C_2+x) = \emptyset )= \p(X\cap C_1 = \emptyset) \p(X\cap C_2 = \emptyset) 
 \end{equation}
 for all $ C_1,C_2 \in \mathcal{C}(\FdAp)$. In fact, \eqref{cond:mixingproperty} is a necessary condition, see  \cite[Theorem 9.3.2]{bo:SW08}, which is used in \cite[Theorem 10.5.1]{bo:SW08} to prove the mixing property of Poisson--Voronoi tessellation. Although we do not follow this approach to prove the mixing property of random Laguerre tessellation -- presenting instead shorter, albeit very non-informative proof -- we do fundamentally expand the approximation arguments from \cite[Section 10.5]{bo:SW08} in the following section to prove the last part of Theorem \ref{thm:Main}. 

\subsection{Proof of Theorem \ref{thm:Main} (c)} \label{subsec:Proofc}

Recall Definition \ref{def:MixTess}~(iii) of an $\alpha$-mixing random tessellation. The coefficient $\alpha(\sigma_X(\tailFnula{a}),\sigma_X(\tailFnula{-b}))$ can be in fact computed as a supremum over smaller sets, since the generating sets $\tailFnula{a}$ and $\tailFnula{-b}$ are semi-algebras in $\mathcal{F}(\FdAp)$, see \cite[Lemma 2.2.2 (a)]{bo:SW08}. Let $\tailF{a}^1$ and $\tailF{-b}^1$ denote the algebras generated by $\tailFnula{a}$ and $\tailFnula{-b}$, \ie the sets of all finite disjoint unions of sets from $\tailFnula{a}$ and $\tailFnula{-b}$. 
\begin{lemma}\label{lemma:alphamixingaltdef}
    Let $X$ be a random tessellation and $a,b > 0$, then
\begin{equation*} 
    \alpha(\sigma_X(\tailFnula{a}),\sigma_X(\tailFnula{-b}))= \sup_{\mathsf{A} \in \tailF{a}^1,  \mathsf{B} \in \tailF{-b}^1} \abs{\p(X \in \mathsf{A}, X \in \mathsf{B}) - \p(X \in \mathsf{A}) \p(X \in \mathsf{B})}.
\end{equation*}
\end{lemma}

\begin{proof} 
The claim can be proved analogously as \cite[Lemma 9.3.1]{bo:SW08}.
\end{proof}

To prove the last implication of Theorem \ref{thm:Main}, we start with three preparatory lemmas. Their main purpose is to show that for tempered configurations, the existence of a very large cell entails the presence of a proportionally large vacant region. Since the nucleus may lie outside (and in principle very far away from) its cell, we must to the largeness of the cell include, in some sense, the distance of its nucleus. We treat separately the cells around the origin and those located far away from it.

First, we introduce three special subsets of admissible Laguerre generators. Let $r \geq 1$, $K > 0$ and define 
\begin{align}
    \begin{split} \label{def:NotationSetsE}
        E^1_{r,K} &\coloneqq \{\varphi: \exists L \in L(\varphi) \text{ s.t. } L \cap rB^d \neq \emptyset \text{ and } (x_L,m_L) \in (KrB^d\times\R)^c \}, \\
         E^2_{r,K} &\coloneqq \{\varphi: \exists L \in L(\varphi) \text{ s.t. } L \cap (KrB^d)^c \neq \emptyset \text{ and } (x_L,m_L) \in rB^d\times\R \},\\
         E^3_{r,K}&\coloneqq \{\varphi:\exists L \in L(\varphi) \text{ s.t. } L \cap rB^d \neq \emptyset \text{ and } L \not\subseteq(K+1)rB^d\},
    \end{split}
\end{align}
where $(x_L,m_L)$ denotes the generator of the cell $L$. Recall the notation $\hatMsetl$ introduced in \eqref{def:hatMsetl} and $k(l)$ from Lemma \ref{lemma:PropertiesTempered}.

\begin{lemma} \label{LBlemma:BoundErHat} 
    Let $l \in \n$ be large enough so that $c_l\coloneqq \sqrt{3k(l)^2+2k(l)-7}-1 > 0$ and choose $K_1(l)\geq 2k(l)+3$. Then for any $r \geq 1$,
    \begin{equation*}
        \varphi \in \hatMsetl \cap E^1_{r,K_1(l)} \implies  \varphi(rc_lB^d \times (-\infty,r^2]) = 0.
    \end{equation*}
\end{lemma}

\begin{proof}
    Assume that $\varphi \in \hatMsetl \cap E^1_{r,K_1(l)}$  and let $L$ be the prescribed cell. Then there exists $z \in rB^d \cap L$ and therefore 
   $\norm{y-z}^2+t \geq \norm{x_L - z}^2+m_L$  holds for all  $(y,t) \in \varphi$. 
   
   Since $(x_L,m_L) \in (K_1(l)rB^d \times \R)^c$ and $K_1(l)r \geq (2k(l)+3)r \geq 2\lceil k(l)r\rceil+1$, the property \eqref{Bound-farpoints} implies that for any $(y,t) \in \varphi$,
    \begin{align*}
        (\norm{y}+\norm{z})^2 +t \geq \norm{y-z}^2 + t &\geq \norm{x_L - z}^2-(\norm{x_L}-k(l)r)^2 \\
        &\geq -2\norm{x_L}\norm{z} + \norm{z}^2+2k(l)r\norm{x_L}-(k(l)r)^2\\
        &\geq -2\norm{x_L}r +2k(l)r\norm{x_L}-(k(l)r)^2\\
        &> 2r^2(2k(l)+3) (k(l)-1) -(k(l)r)^2\\
        &=r^2 (3k(l)^2+2k(l)-6).
    \end{align*}
 Therefore, any $(y,t) \in \varphi$ such that $t \leq r^2$ must satisfy
    \begin{equation*}
        \norm{y} > r \sqrt{3k(l)^2+2k(l)-7} - \norm{z} \geq rc_l,
    \end{equation*}
    and consequently $\varphi(rc_lB^d \times (-\infty,r^2]) = 0$, which completes the proof. 
\end{proof}

 Denote by $B(x,r,R)\coloneqq B(x,R) \setminus B(x,r)$ the annulus with centre $x$ and radii $r< R$. 

\begin{lemma} \label{LBlemma:BoundErTilde}
    Let ${A_1,\dots, A_{J_d}}$ be a partition of $\rd$, where each set $A_i$,  $i = 1, \dots, J_d$, is a closed cone with a vertex at the origin such that $\dotp{u}{v} \geq \frac{1}{2}\norm{u}\norm{v}$ holds for all $u, v \in A_i$. 
    Let $l \in \n$ and fix $K_2(l) > 4l^{\frac{2}{d+\delta}}+10$ and denote $\alpha(l)\coloneqq \big(K_2(l) - 4l^{\frac{2}{d+\delta}}-1\big)^{\frac{1}{2}}$. Then for any $r \geq 1$, 
    \begin{equation*}
        \varphi \in  \hatMsetl \cap E^2_{r,K_2(l)} \implies \exists i \in \{1,\dots,J_d\} \text{ such that } \varphi((A_i \cap B(0,3r,\alpha(l)r)) \times (-\infty,r^2]) = 0. 
    \end{equation*}
\end{lemma}

Notice that $\alpha(l) > 3$, so the annulus is well-defined. 

\begin{proof}
    Assume that $\varphi \in \hatMsetl \cap E^2_{r,K_2(l)}$ and $L \in L(\varphi)$ is the cell satisfying
    \begin{itemize}
        \item[(i)] $L \cap (K_2(l)rB^d)^c \neq \emptyset$,
        \item[(ii)] $(x_L,m_L) \in rB^d\times\R$.
    \end{itemize}
    Take $z \in L \cap (K_2(l)rB^d)^c$ and let $i \in \{1,\dots,J_d\}$ be such that $z \in A_i$. Assume for contradiction that there exists $(y,m_y)\in \varphi \cap ((A_i \cap B^d(0,3r,\alpha(l)r)) \times (-\infty,r^2])$, then 
    \begin{equation} \label{auxbound1}
        \norm{y-z}^2+r^2 \geq \norm{y-z}^2+m_y \geq \norm{x_L-z}^2+m_L.
    \end{equation}
    Therefore, using $(\norm{y}-2\norm{x_L})\norm{z} > K_2(l)r^2$ and $y,z \in A_i$, we can write 
    \begin{align*}
       \norm{y}^2-K_2(l)r^2 &>\norm{y}^2-\norm{y}\norm{z} + 2\norm{x_L}\norm{z} 
       \geq\norm{y}^2-2\dotp{y}{z}-\norm{x_L}^2 + 2\dotp{x_L}{z}\\
        &= \norm{y-z}^2 -\norm{x_L-z}^2 
        \overset{\eqref{auxbound1}}{\geq} m_L - r^2 
        \overset{\eqref{Bound-anypoints}}{>} -l^{\frac{2}{d+\delta}}(\norm{x_L}+1)^{\frac{2d}{d+\delta}}-r^2 \\
        &\overset{(ii)}{\geq} -l^{\frac{2}{d+\delta}}(2r)^{\frac{2d}{d+\delta}}-r^2         \geq -(4l^{\frac{2}{d+\delta}}+1)r^2,
    \end{align*}
    which is in contradiction with $\norm{y}\leq \alpha(l) r  =  \big(K_2(l)-4l^{\frac{2}{d+\delta}}-1\big)^{\frac{1}{2}}r$.
\end{proof}

In the next lemma, use the notation $rB_j\coloneqq B(r x, ra)$ for a general ball $B_j = B(x,a)$, $x \in \rd$, $a \geq 0$ and $r > 0$. 

\begin{lemma} \label{lemma:boundEr}
    Let $l \in \n$, $K_1(l) \geq 2k(l)+3$ and define $C_l\coloneqq 2l^{\frac{1}{d+\delta}}K_1(l)$, $\beta_l\coloneqq \sqrt{C_l^2-1}-1$ and $K_3(l)\coloneqq \sqrt{8}C_l$. Then for any $r\geq 1$,
    \begin{align*}  
    \begin{split}
    \varphi \in E^3_{r,K_3(l)} &\cap \hatMsetl \cap (E^1_{r,K_1(l)})^c \\
    &\implies \varphi(\beta_lrB^d\times(-\infty,r^2])=0 \text{ or } \exists j \in \{1,\dots,\theta_l\} \text{ such that } \varphi(\beta_lrB_j\times(-\infty,r^2])=0 ,
    \end{split}
\end{align*} 
    where $\theta_l$ is the minimal number of balls with unit radius needed to cover the ball $(K_3(l)+1)B^d$ and $B_1,\dots,B_{\theta_l}$ is the minimal cover.
\end{lemma}

\begin{proof}
    Assume that $\varphi \in E^3_{r,K_3(l)} \cap \hatMsetl \cap (E^1_{r,K_1(l)})^c$. Then there exists $L \in L(\varphi)$ such that 
    \begin{itemize}
        \item[(i)] $\exists w \in L \cap rB^d$ and $\exists z \in L$ with $\norm{z} = (K_3(l)+1)r$,
        \item[(ii)] $x_L \in K_1(l)rB^d$ and $m_L > -l^{\frac{2}{d+\delta}}(\norm{x_L}+1)^\frac{2d}{d+\delta} \geq -l^{\frac{2}{d+\delta}}(K_1(l)r+1)^2$ due to \eqref{Bound-anypoints}.
       \end{itemize}
Particularly, we must have $\max{\{\norm{w-x_L}, \norm{z-x_L}\}}   \geq \frac{K_3(l)}{2}r$. 

Assume first that $\norm{w-x_L} \geq \frac{K_3(l)}{2}r$, then for any $(y,t) \in \varphi$, 
\begin{align*}
    \norm{w-y}^2+t &\geq \norm{w-x_L}^2+m_L \\
    &\overset{(ii)}{>} \frac{K_3(l)^2}{4}r^2 -l^{\frac{2}{d+\delta}}(K_1(l)r+1)^2 
    \geq \left(\frac{K_3(l)^2}{4} -4l^{\frac{2}{d+\delta}}K_1(l)^2 \right)r^2. 
\end{align*}
Plugging in the definition of $K_3(l)$, we get $\norm{y-w}^2 + t > C_l^2 r^2$ for any $(y,t) \in \varphi$. Therefore, by the triangle inequality any $(y,t) \in \varphi\cap (\rd \times (-\infty,r^2])$ must satisfy $\norm{y} > (\sqrt{C_l^2-1}-1) r= \beta_l r$.

On the other hand, if we assume that $\norm{z-x_L} \geq \frac{K_3(l)}{2}r$ holds, we get by the same arguments that $\norm{z-y}^2+t > C_l^2 r^2$ for any $(y,t) \in \varphi$. Taking $B_1, \dots, B_{\theta_l}$ the unit volume balls that cover $(K_3(l)+1)B^d$, and $j$ such that $z \in rB_j$, it must hold that $y \notin \beta_lrB_j$ for all $(y,t) \in \varphi\cap (\rd \times (-\infty,r^2])$. 

Altogether we find that the desired implication holds. 
\end{proof}

The following observation shows that for a stationary marked point process, the probability that we observe a growing empty space converges to 0. 

\begin{lemma} \label{lemma:EmptySetLimit}
    Let $\ppp$ be a stationary marked point process on $\rd \times \R$ such that it is \as non-empty. Then for any $x \in \rd$ and any $a > 0$, we have
    \begin{equation*}
        \lim_{r \rightarrow \infty}\p(\ppp(B(rx,ra) \times (-\infty,r^2])=0) =0. 
    \end{equation*}
\end{lemma}

\begin{proof}First assume that $x = 0$. Then 
\begin{align*}
    \lim_{r \rightarrow \infty}\p(\ppp(B(0,ra) \times (-\infty,r^2])=0) = \p(\ppp(\rd \times \R)=0) = 0. 
\end{align*}
    For general $x \in \rd$, we get $\p(\ppp(B(rx,ra) \times (-\infty,r^2])=0) = \p(\ppp(B(0,ra) \times (-\infty,r^2])=0)$ for any $r > 0$ from the stationarity of $\ppp$, which finishes the proof. 
\end{proof}

This allows us to find suitable constants $K_1,K_2,K_3, r_0 \geq 1$ for which the probabilities of the three events \eqref{def:NotationSetsE} are arbitrarily small.

\begin{lemma} \label{lemma:ProbOfErEvents}
    Let $\varepsilon > 0$ and assume that $\ppp$ is a stationary marked point process on $\rd \times \R$ such that $\p(\hat{\ppp} \in \Msettemp) = 1$. Then there exist constants $K_1,K_2,K_3, r_0 \geq 1$ such that for all $r \geq r_0$,
    \begin{equation*} 
        \max \{ \p(\ppp \in E^1_{r,K_1}),\p(\ppp \in E^2_{r,K_2}),\p(\ppp \in E^3_{r,K_3})\} < \varepsilon.
    \end{equation*}
\end{lemma}

\begin{proof}
    Let $\varepsilon > 0$. Then we can find $l_0 \in \n$ such that for any $l \geq l_0$ we get $\p(\ppp\in \hatMsetl)\geq 1-\frac{\varepsilon}{3}$. Recall the notation $c_l$ defined in Lemma \ref{LBlemma:BoundErHat} and fix $l \geq l_0$ large enough so that $c_l > 0$. Choose $K_1 \geq 2k(l)+3$, then by Lemma \ref{LBlemma:BoundErHat}, we get for any $r \geq 1$,
    \begin{equation*}
        \p(\ppp \in E^1_{r,K_1}) \leq \p(\ppp \in E^1_{r,K_1}\cap \hatMsetl) + \frac{\varepsilon}{3} \leq \p(\ppp(B(0,rc_l)\times(-\infty,r^2])=0) + \frac{\varepsilon}{3}.
    \end{equation*}
    Using Lemma \ref{lemma:EmptySetLimit}, there exists $r_1$ such that for any $r \geq r_1$ we have $\p(\ppp(B(0,rc_l)\times(-\infty,r^2])=0) < \frac{\varepsilon}{3}$.  
    
Furthermore, put $K_3\coloneqq K_3(l)$ as defined in Lemma \ref{lemma:boundEr} with our fixed $K_1$ as $K_1(l)$. Then by Lemma \ref{lemma:boundEr} for any $r \geq r_1$,
\begin{align*}
    \p(\ppp \in E^3_{r,K_3}) &\leq \p(\ppp \in E^3_{r,K_3}\cap \hatMsetl \cap (E^1_{r,K_1})^c) + \p(\ppp \in E^1_{r,K_1}\cap \hatMsetl)+ \frac{\varepsilon}{3} \\
    &\leq \p(\ppp(B(0,\beta_l r)\times(-\infty,r^2])=0) + \sum_{i=1}^{\theta_l} \p(\ppp(B(\beta_l r x_i,\beta_l r)\times(-\infty,r^2])=0) + \frac{2\varepsilon}{3},
\end{align*}
where $\beta_l$, $\theta_l$ are defined in Lemma \ref{lemma:boundEr} and $x_i$ are suitable points in $\rd$. Again, using Lemma \ref{lemma:EmptySetLimit}, there exists $r_3 \geq 1$ such that for any $r \geq r_3$ we have 
\begin{align*}
    \p(\ppp(B(0,\beta_l r)\times(-\infty,r^2])=0) + \sum_{i=1}^{\theta_l} \p(\ppp(B(\beta_l r x_i,\beta_l r)\times(-\infty,r^2])=0) < \frac{\varepsilon}{3}. 
\end{align*}
To bound the last probability, recall the notation of Lemma \ref{LBlemma:BoundErTilde}. Fix $K_2 > 4l^{\frac{2}{d+\delta}} + 10$ for our fixed $l$ and realize that there exist $z_1\dots,z_{J_d} \in \rd$ and $a> 0$ such that $B(r z_i, r a) \subseteq A_i \cap B(0,3r,\alpha(l)r)$, where $A_i$ and $\alpha(l)$ are from Lemma \ref{LBlemma:BoundErTilde}. Therefore, for any $r \geq 1$, 
\begin{align*}
    \p(\ppp \in E^2_{r,K_2}) \leq \p(\ppp \in E^2_{r,K_2}\cap \hatMsetl) + \frac{\varepsilon}{3} \leq \sum_{i = 1}^{J_d} \p (\ppp(B(r z_i, r a) \times (-\infty,r^2]) = 0) + \frac{\varepsilon}{3}.
\end{align*}
Again Lemma \ref{lemma:EmptySetLimit} implies that there exists $r_2\geq 1$ such that for all $r \geq r_2$, 
\begin{equation*}
    \sum_{i = 1}^{J_d} \p (\ppp(B(r z_i, r a) \times (-\infty,r^2]) = 0) < \frac{2\varepsilon}{3}. 
\end{equation*} Putting $r_0 \coloneqq \max \{r_1,r_2,r_3\}$ finishes the proof.     
\end{proof}

Let $\varphi \subseteq\rd \times \R$ be such that $L(\varphi)$ is a tessellation. Define the set of neighbours of the point $\bfx\in \varphi$ in $L(\varphi)$ as 
\begin{equation*}
\mathbf{N}(\bfx, \varphi) \coloneqq \{{\bf y} \in \varphi\setminus \{\bfx\}: L(\bfx, \varphi) \cap L({\bf y}, \varphi)\neq \emptyset\}.    
\end{equation*}

Next, we show that for a suitable subset of generators, it is enough to know the generators close to or far away from the origin to be able to reconstruct cells close to (see Lemma \ref{lemma:NeighborsA}) or far away (see Lemma \ref{lemma:NeighborsB}) from the origin. 

\begin{lemma} \label{lemma:NeighborsA}
    Let $a > 0$ and denote $\Hat{\Theta}_{r,K_3,K_1} \coloneqq (E^3_{r,K_3} \cup E^1_{(K_3+1)r,K_1})^c$ for $r \geq \max \{a,1\}$ and $K_1,K_3  > 0$. Then for any $\mathsf{A} \in \tailF{a}^1$ we have
    \begin{align} \label{BoundForNeighborsA}
        \Hat{\Theta}_{r,K_3,K_1} \cap \{\varphi: L(\varphi) \in \mathsf{A}\} = \Hat{\Theta}_{r,K_3,K_1} \cap \{\varphi: L(\varphi \cap (K_1(K_3+1)rB^d \times \R)) \in \mathsf{A}\}.
    \end{align}
\end{lemma}

\begin{proof}
    It is enough to show the equality for sets $\mathsf{A}$ from the generating semi-algebra $\tailFnula{a}$, \ie for $\mathsf{A} = \mathcal{F}^{C_0}_{C_1,\dots,C_k}$ where $k \in \n$ and $C_0,\dots,C_k \in \mathcal{C}(\FdAp)$, $C_i \subseteq \mathcal{F}_{aB^d}$, $i = 0,\dots,k$. For such $\mathsf{A}$ we have
    \begin{equation*}
        \{\varphi: L(\varphi) \in \mathsf{A}\} = \{\varphi: L(\varphi) \cap C_0 = \emptyset\} \cap \bigcap_{i = 1}^k \{\varphi: L(\varphi) \cap C_i \neq \emptyset\}.
    \end{equation*}
   Therefore, to show \eqref{BoundForNeighborsA}, it is enough to show that for $\varphi \in \Hat{\Theta}_{r,K_3,K_1}$ and $C\in \mathcal{C}(\FdAp)$, $C \subseteq \mathcal{F}_{aB^d}$, 
    \begin{equation} \label{equalityofcells}
        L(\varphi) \cap C =  L(\varphi \cap (K_1(K_3+1)rB^d \times \R)) \cap C.
    \end{equation}
    
     Take $\varphi \in \Hat{\Theta}_{r,K_3,K_1}$ and let $L \in L(\varphi)\cap C$. Then, thanks to the choice of $r$, we have $L \cap rB^d \neq \emptyset$. It follows that $L \subseteq (K_3+1)rB^d$ and $(x_L,m_L) \in K_1(K_3+1)rB^d \times \R$ from the definition of sets $E^1_{\cdot,\cdot}$ and $E^3_{\cdot,\cdot}$. This implies that any $\bfy \in \mathbf{N}((x_L,m_L),\varphi)$ satisfies $L(\bfy, \varphi) \cap (K_3+1)rB^d \neq \emptyset$ and therefore $\mathbf{N}((x_L,m_L),\varphi) \subseteq\varphi \cap (K_1(K_3+1)rB^d \times \R)$. This implies that for any $\bfx \in \varphi$ such that $L(\bfx,\varphi)\in L(\varphi) \cap C$, we have
     \begin{equation*}
         L(\bfx,\varphi) = L(\bfx,\varphi \cap (K_1(K_3+1)rB^d \times \R) ),
     \end{equation*}
which finishes the proof of \eqref{equalityofcells} and consequently of \eqref{BoundForNeighborsA}. 
\end{proof}

\begin{lemma} \label{lemma:NeighborsB}
    Let $b > 0$ and let $\mathsf{B} \in \tailF{-b}^1$. Let $K_3,K_2 > 0$ and $r \geq 1$ be such that $b \geq (K_3+1)K_2r$ and denote $\Tilde{\Theta}_{r,K_3,K_2}\coloneqq (E^3_{K_2r,K_3} \cup E^2_{r,K_2})^c$. Then 
    \begin{equation} \label{BoundForNeighborsB}
        \Tilde{\Theta}_{r,K_3,K_2} \cap \{\varphi: L(\varphi) \in \mathsf{B}\} = \Tilde{\Theta}_{r,K_3,K_2} \cap \{\varphi: L(\varphi \cap ((rB^d)^c \times \R)) \in \mathsf{B}\}.
    \end{equation}
\end{lemma}

\begin{proof}
     Again, it is enough to consider $\mathsf{B}$ from the generating semi-algebra $\tailFnula{-b}$. Let $k \in \n$ and the sets $C_0,\dots,C_k \in \mathcal{C}(\FdAp)$, $C_i \subseteq \mathcal{F}_{(bB^d)^c}$, $i = 0,\dots,k$,  be such that $\mathsf{B} = \mathcal{F}^{C_0}_{C_1,\dots,C_k}$. As in Lemma \ref{lemma:NeighborsA}, to prove \eqref{BoundForNeighborsB} it suffices to show that for any $\varphi \in \Tilde{\Theta}_{r,K_3,K_2}$ and $C \in \mathcal{C}(\FdAp)$, $C \subseteq \mathcal{F}_{(bB^d)^c}$, we have
\begin{equation} \label{equalityofcellsB}
           L(\varphi) \cap C =  L(\varphi \cap ((rB^d)^c \times \R)) \cap C.
    \end{equation}

Take $\varphi \in \Tilde{\Theta}_{r,K_3,K_2}$ and $L \in L(\varphi) \cap C$, then $L \cap (\rd \setminus (K_3+1)K_2rB^d) \neq \emptyset$. This, together with $\varphi \in (E^3_{K_2r,K_3})^c$, implies that $L \subseteq \rd \setminus K_2rB^d$. Since $\varphi \in (E^2_{r,K_2})^c$, it follows that $\bfx_L \in (rB^d)^c \times \R$. Furthermore, for all $\bfy \in \mathbf{N}(\bfx_L,\varphi)$ we get $L(\bfy,\varphi) \cap (K_2rB^d)^c \neq \emptyset$ and therefore also $\bfy \in (rB^d)^c \times \R$. 

Altogether, we have shown that $L(\bfx,\varphi) =  L(\bfx, \varphi \cap ((rB^d)^c \times \R))$ holds for any $\bfx \in \varphi$ such that $L(\bfx, \varphi) \in L(\varphi) \cap C$, which finishes the proof of \eqref{equalityofcellsB}. 
\end{proof}

Now we are ready to prove our main result. The main idea is to use Lemmas \ref{lemma:ProbOfErEvents}, \ref{lemma:NeighborsA} and \ref{lemma:NeighborsB} to approximate the probability $\p(L(\ppp) \in \mathsf{A})$ by $\p(L(\ppp \cap (RB^d \times \R)) \in \mathsf{A})$ for some suitable constant $R$ for any $\mathsf{A} \in  \tailF{a}^1$ and similarly for $\mathsf{B} \in  \tailF{-b}^1$. This allows us to control the $\alpha$-mixing coefficient of $L(\ppp)$ by the $\alpha$-mixing coefficient of $\ppp$ up to some $\varepsilon$ term, which can be made arbitrarily small. 

\begin{proof}[Proof of Theorem \ref{thm:Main} (c)]First, thanks to part (a) of Theorem \ref{thm:Main}, $L(\ppp)$ is \as a tessellation of $\rd$. Our aim is to prove that $\lim_{b \rightarrow \infty} \alpha(\sigma_X(\tailFnula{a}),\sigma_X(\tailFnula{-b})) = 0$ holds for any $a > 0$.
    
Fix $a > 0$ and let $\varepsilon > 0$. Recall the notation $\hat{\ppp}$ from \eqref{def:hattrans}. By the same argument as in the proof of part (a), $\p(\hat{\ppp} \in \Msettemp)  = 1$. Let $K_1,K_2,K_3 \geq 1$ and $r_0 \geq 1$ be the constants from  Lemma \ref{lemma:ProbOfErEvents} and denote $K_{i,3}\coloneqq K_i(K_3+1)$, $i = 1,2$, then 
\begin{itemize}
    \item[(i)] for $a_0 \coloneqq \max \{r_0,a\}$ we get $\p\big(\ppp \in \Hat{\Theta}^c_{a_0,K_3,K_1}\big) = \p\big(\ppp \in E^3_{a_0,K_3}\cup E^1_{(K_3+1)a_0,K_1}\big) < 2\varepsilon$,
    \item[(ii)] for $b \geq K_{2,3}r_0$ we get $ \p\big(\ppp \in \Tilde{\Theta}^c_{\frac{b}{K_{2,3}},K_3,K_2}\big) = \p\big(\ppp \in E^3_{\frac{b}{K_3+1},K_3}\cup E^2_{\frac{b}{K_{2,3}},K_2}\big) < 2\varepsilon$.
\end{itemize}

Take $\mathsf{A} \in \tailF{a}^1$. Let $b \geq K_{2,3}r_0$ and take $\mathsf{B} \in \tailF{-b}^1$. 
Using Lemmas \ref{lemma:NeighborsA} and \ref{lemma:NeighborsB}, we can bound
\begin{align*}
    &\abs{\p\big(L(\ppp) \in \mathsf{A}\big) - \p\big(L(\ppp \cap (K_{1,3}a_0B^d \times \R)) \in \mathsf{A}\big)}  \leq \p\big(\ppp \in \Hat{\Theta}^c_{a_0,K_3,K_1}\big), \\
    &\abs{\p\big(L(\ppp) \in \mathsf{B}\big) - \p\big(L(\ppp \cap ((\textstyle \frac{b}{K_{2,3}}B^d)^c \times \R)) \in \mathsf{B}\big)}  \leq \p\big(\ppp \in \Tilde{\Theta}^c_{\frac{b}{K_{2,3}},K_3,K_2}\big), \\
    &\abs{\p\big(L(\ppp) \in \mathsf{A}, L(\ppp) \in \mathsf{B}\big) - \p\big(L(\ppp \cap (K_{1,3}a_0B^d \times \R)) \in \mathsf{A}, L(\ppp \cap ((\textstyle \frac{b}{K_{2,3}}B^d)^c \times \R)) \in \mathsf{B}\big)} \\
    & \hspace{8cm}\leq \p\big(\ppp \in \Tilde{\Theta}^c_{\frac{b}{K_{2,3}},K_3,K_2} \cup \Hat{\Theta}^c_{a_0,K_3,K_1}\big).
\end{align*}
Together with (i), (ii) and Theorem \ref{thm:MeasurabilityOfL}, these bounds imply that for $b$ large enough
\begin{align*}
    &\abs{\p\big(L(\ppp) \in \mathsf{A}, L(\ppp) \in \mathsf{B}\big) - \p\big(L(\ppp)\in \mathsf{A}\big)\p\big(L(\ppp)\in \mathsf{B}\big)} \\
    &\hspace{1cm}\leq  \big\lvert \p\big(L(\ppp \cap (K_{1,3}a_0B^d \times \R)) \in \mathsf{A}, L(\ppp \cap ((\textstyle \frac{b}{K_{2,3}}B^d)^c \times \R)) \in \mathsf{B}\big) \\
    & \hspace{3cm}- \p\big(L(\ppp \cap (K_{1,3}a_0B^d \times \R)) \in \mathsf{A}\big)\p\big(L(\ppp \cap ((\textstyle \frac{b}{K_{2,3}}B^d)^c \times \R)) \in \mathsf{B}\big) \big\rvert + 8 \varepsilon \\
    &\hspace{1cm} \leq \alpha_\ppp \big(K_{1,3}a_0,\infty;\textstyle \frac{b}{K_{2,3}}-K_{1,3}a_0\big) + 8\varepsilon.
\end{align*}

Since the choice of $\mathsf{A}$ and $\mathsf{B}$ was arbitrary and thanks to Lemma \ref{lemma:alphamixingaltdef}, we get that for any $b$ large enough
\begin{equation*}
    \alpha(\sigma_X(\tailFnula{a}),\sigma_X(\tailFnula{-b})) \leq  \alpha_\ppp \big(K_{1,3}a_0,\infty;\textstyle \frac{b}{K_{2,3}}-K_{1,3}a_0\big) + 8\varepsilon.
\end{equation*}
Taking $\limsup_{b \rightarrow \infty}$ and $\lim_{\varepsilon \rightarrow 0}$ finishes the proof. 
\end{proof}

\section{Examples}\label{sec:Examples}

We conclude by discussing some examples of admissible Laguerre generators $\ppp$. First, Section \ref{subsec:PoissonEx} presents the Poisson--Laguerre tessellation, for which the moment assumption on the typical mark can be made optimal. We then briefly discuss several non-Poissonian examples in Section \ref{subsec:OtherEx}.

\subsection{Poisson--Laguerre tessellation} \label{subsec:PoissonEx}

Let $\ppp$ be a stationary marked Poisson point process (\ie a Poisson point process that is also marked) with the intensity measure $\gamma \lambda^d \otimes \markdist$, where $\gamma > 0$  is the intensity and $\markdist$ is a probability measure on $\R$. It is a well-known fact that $\ppp$ is stationary and ergodic and therefore it is enough to require that $M \sim \markdist$ satisfies \eqref{as:momentM_} for $L(\ppp)$ to be \as a tessellation, which is called Poisson--Laguerre tessellation.

However, in the Poisson case, we can in fact weaken the moment condition on the mark distribution, since assumption (R1) can be verified without the help of tempered configurations. This result is again in accordance with \cite[Theorem 4.1]{ar:LZ08} and we only briefly sketch the proof here.

\begin{lemma} \label{lemma:ExPoisLagTess}
    Let $\ppp$ be a marked Poisson point process on $\rd \times \R$ with intensity measure $\gamma \lambda^d \otimes \markdist$, where $\gamma > 0$  and $\markdist$ is a probability measure on $\R$. Then $L(\ppp)$ is \as a tessellation of $\rd$ if and only if the typical mark $M \sim \markdist$ satisfies 
    \begin{equation} \label{ass:FiniteDthMoment}
    \e M_-^\frac{d}{2}< \infty.
\end{equation}

\end{lemma}

\begin{proof}
   Assume that \eqref{ass:FiniteDthMoment} holds. According to Proposition \ref{prop:Tessellation}, it is enough to show that then $$\p(\ppp \text{ satisfies (R1), (R2)})=1.$$
   Assumption (R2) follows from the stationarity of $\ppp'$. To show that (R1) holds, realize that for any $t \in \z$ and $n \in \n$ the assumption \eqref{ass:FiniteDthMoment} implies $\e \sum_{(x,m)\in \ppp }\ind \{B(0,n)\cap B(x,\sqrt{(t-m)_{+}}) \neq \emptyset\} < \infty$ and therefore 
   \begin{align*}
   1 &= \p\bigg(\abs{\{(x,m)\in \ppp : B(0,n)\cap B(x,\sqrt{(t-m)_{+}}) \neq \emptyset\}} < \infty\bigg) \\
   &\leq \p\bigg(\forall 
   y \in B(0,n): \abs{\{(x,m)\in \ppp : y \in B(x,\sqrt{(t-m)_{+}})\}} < \infty \bigg).
   \end{align*}
   Taking $n \rightarrow \infty$ on both sides and $\cap_{t \in \z}$ finishes the proof. 
   
Now assume that $\e M_-^\frac{d}{2} = \infty$. Then, using \cite[Theorem 3.2.4]{bo:SW08},
\begin{align*}
    \p(\inf_{\bfx \in \ppp} \pd(0,\bfx)=-\infty) &= \lim_{t \rightarrow-\infty} \p(\inf_{\bfx \in \ppp} \pd(0,\bfx)\leq t) = \lim_{t \rightarrow-\infty} \left(1-\e \prod_{\bfx \in \ppp} \ind \left\{\pd(0,\bfx) > t\right\}\right)\\
    &= \lim_{t \rightarrow-\infty} \left(1- \mathrm{e}^{-\gamma \lambda^d(B^d) \int_{-\infty}^t(t-m)^{d/2}\,\dx \markdist(m) }\right) = 1.
\end{align*}
In other words, there almost surely does not exist a minimizer (over $\ppp$) of the power distance to the origin $0$, which means that $L(\ppp)$ is not space filling. 
\end{proof}

Let us note that Lemma \ref{lemma:ExPoisLagTess} is also in accordance with \cite[Theorem 3.3 (iii)]{ar:GWL25}, which is stated as an implication and for absolutely continuous mark distributions. Generalization of the proof therein to a setting with a general mark distribution should also be straightforward. 

\subsection{Non-Poissonian Examples} \label{subsec:OtherEx}

It follows from Theorem \ref{thm:Main} (a) that to obtain an admissible random Laguerre generator, we seek ergodic marked point processes satisfying the moment condition \eqref{as:momentM_}. The rudimentary family of models is given by independently marked point processes. 

We say that a marked point process $\ppp$ on $\rd \times \R$ is an independently marked point process if the marks form a sequence of independent identically distributed random variables that is independent of the ground process $\ppp'$. We refer to the common distribution of the marks as the mark distribution. In the stationary case, this distribution coincides with the stationary mark distribution. To check the ergodicity or mixing property of an independently marked point process, it suffices to verify the corresponding property for the ground process.

\begin{lemma} \label{lemma:MixProgIMPP}
Let $\ppp$ be a stationary independently marked point process with intensity measure $\gamma \lambda^d \otimes \markdist$ for some $\gamma > 0$ and mark distribution $\markdist$. Assume that the ground process $\ppp'$ is ergodic or mixing, then $\ppp$ has the same property.     
\end{lemma}

\begin{proof}
    We consider the proof of the ergodic property. It follows from \cite[Lemma  12.3.II]{bo:DVJ08b} that it is enough to show 
    \begin{align*} 
    \begin{split}
        \lim_{a \rightarrow \infty}\frac{1}{\lambda^d((-a,a)^d)} \int_{(-a,a)^d} \p(T_z(\ppp)&(B_1\times M_1)=k_1,\ppp(B_2\times M_2)=k_2)\,\dx z 
        \\ &=  \p(\ppp(B_1\times M_1)=k_1)\p(\ppp(B_2\times M_2)=k_2)
        \end{split}
    \end{align*}
for any $k_1,k_2 \in \n_0$, $B_1,B_2 \in \borel(\rd)$ bounded and $M_1,M_2 \in \borel(\R)$. Choose such $k_1$, $k_2$, $B_1$, $B_2$, $M_1$, $M_2$, let $a_0 > 0$ be such that $\operatorname{dist}(B_1,B_2)\leq a_0$ and denote $p_i(j)\coloneqq\binom{j}{k_i}\markdist(M_i)^{k_i}(1-\markdist(M_i))^{j-k_i}$ for $j \geq k_i$, $i = 1,2$. Then we can write, thanks to the independence of the ground process and the sequence of marks, 
\begin{align*}
    \lim_{a \rightarrow \infty}\frac{1}{(2a)^d} &\int_{(-a,a)^d} \p(T_z(\ppp)(B_1\times M_1)=k_1,\ppp(B_2\times M_2)=k_2)\,\dx z \\
    &= \lim_{a \rightarrow \infty}\frac{1}{(2a)^d}  \int_{(-a,a)^d\setminus B(0,a_0)} \sum_{j_i \geq k_i,i=1,2}^\infty p_1(j_1)p_2(j_2)\p(T_z(\ppp')(B_1)=j_1,\ppp'(B_2)=j_2)\,\dx z \\
    &=  \sum_{j_i \geq k_i,i=1,2}^\infty p_1(j_1)p_2(j_2) \lim_{a \rightarrow \infty}\frac{1}{(2a)^d}  \int_{(-a,a)^d}\p(T_z(\ppp') (B_1)=j_1,\ppp'(B_2)=j_2)\,\dx z \\
     &=  \sum_{j_i \geq k_i,i=1,2}^\infty p_1(j_1)p_2(j_2) \p(\ppp' (B_1)=j_1)\p(\ppp'(B_2)=j_2) \\ &= \p(\ppp(B_1\times M_1)=k_1)\p(\ppp(B_2\times M_2)=k_2).
\end{align*}
    The mixing property would be shown analogously. 
\end{proof}

As a first example, consider the Cox process also known as doubly stochastic Poisson process.

\begin{example}\label{ex:Cox}
    Let $\Psi$ be a random measure on $\rd$. We say that a point process $\ppp'$ is a Cox point process on $\rd$ with driving measure $\Psi$, if conditionally on $\Psi$ it is a Poisson point process on $\rd$ with intensity measure $\Psi$ (see \cite[Definition 6.2.I]{bo:DVJ03a}).  It follows from \cite[Proposition 12.3.VII]{bo:DVJ08b} and Lemma \ref{lemma:MixProgIMPP} that an independently marked Cox process $\ppp$ (ground process $\ppp'$ is a Cox process) with stationary ergodic driving measure $\Psi$ and mark distribution $\markdist$ satisfying \eqref{as:momentM_} is an admissible Laguerre generator. Furthermore, if $\Psi$ is mixing random measure, then also $\ppp$ is mixing which implies mixing of $L(\ppp)$. For a discussion about the $\alpha$-mixing property of (unmarked) Cox processes see for example \cite{ar:WG09}. 
\end{example}

The next standard example is the marked cluster process. For simplicity, we present only the special case of independent cluster processes with finite clusters, see \cite[Definition 6.3.I]{bo:DVJ03a} for the general definition. 

\begin{example} \label{ex:Cluster}
    Let $\Psi$ be a simple point process on $\rd$ and let $\{\xi_i\}_{i \in \n}$ be a sequence of iid finite point processes in $\rd$ such that $\xi_1$ is concentrated on some bounded set. Then $\ppp'\coloneqq \sum_{X_i \in \Psi} (\xi_i+X_i)$ is  the cluster process with centre process $\Psi$ and iid clusters $\{\xi_i\}_{i \in \n}$. It follows from \cite[Proposition 12.3.IX]{bo:DVJ08b} and Lemma \ref{lemma:MixProgIMPP} that an independently marked cluster process $\ppp$ with stationary ergodic centre process, iid clusters and mark distribution satisfying \eqref{as:momentM_} is an admissible Laguerre generator. Furthermore, if the centre process is mixing, also $\ppp$ is mixing. If the centre process $\Psi$ is a stationary Poisson point process, the cluster process is also $\alpha$-mixing, see \cite[Remark 6.3]{ar:HM99}.  
\end{example}

Let us briefly mention some other examples without going into the details of their formal definitions. To obtain admissible Gibbs--Laguerre generators, it is enough to realize that by the proof of the existence of an infinite-volume Gibbs process in \cite{ar:RZ20}, we obtain a stationary marked Gibbs point process that is concentrated on the set $\Msettemp$. Therefore, by Lemma \ref{lemma:ExistenceLagTess}, we obtain an admissible Laguerre generator, albeit with non-positive marks only, as the existence theorem in \cite{ar:RZ20} ensures the existence of $\hat{\ppp}$ rather than $\ppp$. Furthermore, suitable independently marked Gibbs processes are $\alpha$-mixing admissible random Laguerre generators, see \cite{ar:H92}.
Other examples of $\alpha$-mixing admissible random Laguerre generators include also dependently thinned (Poisson) point processes, see \cite[Remark 6.3]{ar:HM99}, and suitable subclass of independently marked determinantal point processes, see \cite{ar:BW19, ar:PDL19} for more details. 

Finally, we present an example of a marked point process with some dependency structure in the marking procedure. 

\begin{example} \label{ex:Geostatistical}
Let $\ppp'$ be a simple stationary point process on $\rd$ and let $M=\{M(x): x \in \rd\}$ be a real-valued stationary random field. Assume that $\ppp'$ and $M$ are independent. Then the process $\ppp = \sum_{x \in \ppp'} \delta_{(x,M(x))}$ is called the geostatistically marked point process and we have the following bound
\begin{equation*} 
\alpha(\sigma(\ppp \cap (B_1 \times \R)),\sigma(\ppp \cap (B_2 \times \R))) \leq \alpha(\sigma(\ppp' \cap B_1 ),\sigma(\ppp' \cap B_2)) + \alpha(\sigma(M \cap B_1),\sigma(M \cap B_2)).
\end{equation*}
The proof proceeds exactly as in \cite[Lemma 5.1]{ar:HLS14}, where the result is established for a related $\beta$-mixing coefficient. Therefore, if both $\ppp'$ and $M$ are $\alpha$-mixing (in the same sense as Definition \ref{def:MixingMPP}~(iii)), then $\ppp$ is an $\alpha$-mixing admissible Laguerre generator. 
\end{example}

\subsubsection*{Acknowledgements.} 
ZP was supported by the Czech Science Foundation (project no.\,24-10822S).
MŠP was supported by the Charles University Grant Agency (project no.\,70524) and the Charles University Research Center (program No.\,UNCE/24/SCI/022).

\end{document}